\documentclass[12pt]
{amsart}
\usepackage{amsfonts,amsmath,amssymb,amsthm,graphicx,afterpage,url}
\usepackage{hyperref}
\usepackage{color}

\input{xy}
\xyoption{all}

\def\lk{\mbox{lk}}
\def\st{\mbox{st}}

\def\st{\mathop{\fam0 st}}
\def\lk{\mathop{\fam0 lk}}

\def\R{{\mathbb R}} \def\Z{{\mathbb Z}}

\newcommand{\jonly}[1]{}
\newcommand{\aronly}[1]{#1}

\hypersetup{
  colorlinks   = true, %Colours links instead of ugly boxes
    urlcolor     = blue, %Colour for external hyperlinks
    linkcolor    = blue, %Colour of internal links
    citecolor   = red %Colour of citations
}

    \theoremstyle{theorem}
         \newtheorem{theorem}{Theorem}
         \newtheorem{lemma}[theorem]{Lemma}
         \newtheorem{corollary}[theorem]{Corollary}
         
    \theoremstyle{definition}
         \newtheorem{remark}[theorem]{Remark}

\begin{document}

\title{Short proofs \\ of Tverberg-type theorems for cell complexes}

\author{R. Karasev and A. Skopenkov}

\thanks{R. Karasev: \texttt{http://www.rkarasev.ru/en/about/}.
\newline
A. Skopenkov: \texttt{https://users.mccme.ru/skopenko}.
\newline
We are grateful to S. Avvakumov and the anonymous referee for useful discussions.
\jonly{\newline
In this journal version some remarks (minor or less relevant to this note)
from the arXiv version are omitted.}
}

%keywords

\date{}

\maketitle

\abstract
We present short proofs of Tverberg-type theorems for cell complexes by S. Hasui, D. Kishimoto, M. Takeda, and M. Tsutaya.
One of them states that for any prime power $r$, any complex $X$ topologically homeomorphic to $S^{(d+1)(r-1)-1}$,
and any continuous map $f:X\to\R^d$ there are pairwise disjoint faces $\sigma_1,\ldots,\sigma_r$ of $X$ such that $f(\sigma_1)\cap\ldots\cap f(\sigma_r)\ne\emptyset$.
\endabstract

We present short proofs of Tverberg-type theorems for cell complexes
(\cite[Theorem 1.4 and Corollary 1.5]{HKTT} quoted as Corollary \ref{c:simsph} and Theorem \ref{t:tvco} below).

We abbreviate `finite simplicial complex' to `complex'.
We mostly omit `continuous' for maps and group actions.
Denote by $\Delta_n$ the simplex of dimension $n$.

A map  $f\colon X\to \R^d$ from a complex $X$ is said to be an \emph{almost $r$-embedding} if
\linebreak
$f(\sigma_1)\cap\ldots\cap f(\sigma_r)=\emptyset$ whenever $\sigma_1,\ldots,\sigma_r$ are pairwise disjoint faces of $X$.

The topological Tverberg theorem states that \emph{if $r$ is a prime power, then the simplex $\Delta_{(d+1)(r-1)}$
%of dimension $(d+1)(r-1)$
(or, equivalently, its boundary) has no almost $r$-embeddings to $\R^d$}.
This was proved by B\'ar\'any--Schlossman--Sz\H{u}cs and \"Ozaydin--Volovikov, see surveys \cite{BZ16, Sk16} and the references therein.
(A counterexample for $r$ not a prime power follows from results of \"Ozaydin, Mabillard--Wagner, and Gromov--Blagojevi\'c--Frick--Ziegler; for the details, see the surveys \cite{BZ16, Sk16} and the references therein\aronly{; cf. Remark \ref{r:hktt}.c}.)
The topological Tverberg theorem was generalized as follows.

\begin{corollary}[{\cite[Corollary 1.5]{HKTT}}]\label{c:simsph}
If $r$ is a prime power, then no complex topologically homeomorphic to $S^{(d+1)(r-1)-1}$
%$\partial\Delta_{(d+1)(r-1)}$
has an almost $r$-embeddings to $\R^d$.
\end{corollary}

%Ули рассказал доказательство которое работает для гомологических сфер (и r степени простого).
%Идея в том, что для любой гомологической сферы существует цепное отображение в нее из границы симплекса, которое
%1) отображает вершины в вершины
%2) непересекающиеся грани симплекса в непересекающиеся цепи
%Отображение строится просто по индукции, используя что линки вершин тоже гомологические сферы.
%Далее из того, что препятствие для взрезанной степени границы симплекса не 0 и существования этого отображения делается вывод что препятствие и для взрезанной степени гомологической сферы тоже не 0.
%Потому, что препятствие (ко)гомологическое.

Corollary \ref{c:simsph} proves for prime powers the Tverberg--B\'ar\'any--Kalai conjecture stated in
\cite[before question 1.2]{HKTT}, \cite[Conjecture 1.2]{SZ24}.
For non-prime-powers that conjecture concerns maps \emph{linear} on faces, so it remains open.
The PL version of Corollary \ref{c:simsph} immediately follows from the topological Tverberg theorem and Lemma \ref{l:subdiv}.a below.

%Just as in \cite{HKTT},
Corollary \ref{c:simsph} follows from Theorem \ref{t:tvco} (\cite[Theorem 1.4]{HKTT})
by \cite[Proposition 2.3]{HKTT} (which we do not state).
%In this paper 
We simplify the proof of Theorem \ref{t:tvco} \cite[\S3, \S4]{HKTT}.
%(Concerning a proof of \cite[Proposition 2.3]{HKTT} see Remark \ref{r:hktt}.b.)

Before we state Theorem \ref{t:tvco}, in this paragraph we describe how to naturally invent it.
(This description, and the terminology it uses, is not used in the statements below.)
The idea is to use analogy to the \"Ozaydin--Volovikov proof of the topological Tverberg theorem.
Namely, to use a lemma of \"Ozaydin--Volovikov (quoted below, \cite[Lemma]{vo96}), and mimick the spectral sequence argument proving high enough acyclicity of the total
%relevant configuration
space of a bundle, in the situation when a
%relevant projection
map is not a bundle.
There naturally appears a property of a simplicial sphere (complementary acyclicity defined in \cite[\S1]{HKTT} and below) allowing to carry this argument.
Thus the proof of Theorem \ref{t:tvco} is an easy combination of (non-trivial) known results and methods, see details in  Remark \ref{r:hktt}\aronly{.ab} and footnote \ref{f:spectr}.
%Prop 2.3 is applied to d=(d+1)(r-1)-1 and k=r-1

Let $X$ be a complex.
Let $X-\sigma$ be the subcomplex of $X$ consisting of faces which do not intersect $\sigma$.
This is not to be confused with $X\setminus\sigma$.
We take a prime $a$ and $r=a^m$; we
omit coefficients $\Z_a$
%not $(\Z_a)^m$
from the notation of homology groups.
%, where $a$ is a prime such that $r$ is a power of $a$.

For integers $n\ge-1$ and $s\ge1$ a complex $X$ is said to be

$\bullet$ \emph{$n$-acyclic} if $X$ is non-empty and $\widetilde H_j(X)=0$ for every $j\le n$.

$\bullet$ \emph{$s$-complementary $n$-acyclic} if
for every $i=0,1,\ldots,s$, and pairwise disjoint faces $\sigma_1,\ldots,\sigma_i$ of $X$,
%the sum of whose dimensions does not exceed $d(r-1)$
the complex $X-\sigma_1-\ldots-\sigma_i$ is $(n-\dim\sigma_1-\ldots-\dim\sigma_i)$-acyclic.

Set
$$n:=d(r-1)-1.$$

\begin{theorem}[{\cite[Theorem 1.4]{HKTT}}]\label{t:tvco} If $r$ is a power of a prime, then no $(r-1)$-complementary $n$-acyclic complex has
an almost $r$-embedding to $\R^d$.
%Then for any continuous map $f:X\to\R^d$ there are pairwise disjoint faces $\sigma_1,\ldots,\sigma_r$ of $X$ such that $f(\sigma_1)\cap\ldots f(\sigma_r)\ne\emptyset$.
\end{theorem}

%https://core.ac.uk/download/pdf/82377898.pdf:
%A cW complex is said to be regular if all closed cells are homeomorphic to closed balls En (cf. [16, p. 78]).

Set
\[
X^{\underline r} := \cup \{ \sigma_1 \times \cdots \times \sigma_r
\ : \sigma_i \textrm{ is a face of }X,\ \sigma_i \cap \sigma_j = \emptyset \mbox{ for every }i \neq j \}.
\]
This is a topological analogue of the set of arrangements, so we use the notation analogous to
$x^{\underline r}=x(x-1)\ldots(x-r+1)$.

Let $\pi:X^{\underline r}\to X$ be the restriction of the projection $X^r\to X$.

\begin{proof}[Proof of Theorem \ref{t:tvco}]
The group $G:=\Z_a^m$ acts on $X^{\underline r}$ by permutations of the $r$ points, the permutations corresponding to the action of $G$ on itself by left shifts.
Thus $G$ can be considered as a subgroup of the group of permutation of $r=|G|$ elements.
Hence it is clear and well-known (see e.g. surveys \cite[Lemma 3.9]{BZ16}, \cite[Lemma 2.3]{Sk16}) that

\emph{if there is an almost $r$-embedding $X\to\R^d$, then there are a fixed points free action of $G$ on $S^n$ and a $G$-equivariant map $X^{\underline r}\to S^n$}.

Apply the following result \cite[Lemma]{vo96} for $k=n$ and $Y=X^{\underline r}$.

\emph{Suppose that $G$ acts on $Y$ and on $S^k$ without fixed points, and $Y$ is $k$-acyclic.
Then there does not exist a $G$-equivariant map $Y\to S^k$.}

We see that it suffices to show that \emph{$X^{\underline r}$ is $n$-acyclic.}

We show this by induction on $r$.
The base $r=1$ follows because $X^{\underline 1}=X$.
Let us prove the inductive step.
%The $\pi$-preimage of a point is $C_{r-1}(X-\sigma)$, where $\sigma$ is the minimal face containing the point.
Since $X$ is $(r-1)$-complementary $n$-acyclic, for each face $\sigma$ of $X$ the complex $X-\sigma$ is $(r-2)$-complementary $(n-\dim\sigma)$-acyclic.
Hence by the inductive hypothesis $(X-\sigma)^{\underline{r-1}}$ is $(n-\dim\sigma)$-acyclic.
Now the inductive step holds by the following Lemma \ref{l:crcon} because $X$ is $n$-acyclic.
\end{proof}

\begin{lemma}\label{l:crcon} If $F_\sigma:=(X-\sigma)^{\underline{r-1}}$ is $(n-\dim\sigma)$-acyclic for every face $\sigma$ of a complex $X$,
%with $\dim\sigma\le n+1$,
then $H_s(X^{\underline r})\cong H_s(X)$ for every $s\le n$.
\end{lemma}

\begin{proof} Consider
\footnote{\label{f:spectr} This proof is an exercise on spectral sequences analogous to \cite[\S21.1B]{FF89}. In particular, construction of $E^1$ mimicks construction of $E^1$ for the spectral sequence of a bundle, although $\pi$ need not be even the Serre fibration.
\newline
Lemma \ref{l:crcon} and the construction of $E^1$ in its proof are clarified versions of \cite[Lemma 4.4 Proposition 4.3]{HKTT} (omitting unnecessary hocolimits).
In \cite[\S4]{HKTT} the construction of $E^1$ (sufficient for the main result) was implicitly stated (before Proposition 4.3), and was used to deduce Proposition 4.3 (stated in the hocolimits language unnecessary for the main result).}
%\newline
%The corresponding part \cite[the 2nd sentence in the proof of Lemma 4.4]{HKTT} is unclear because $C_{r-1}(X-\sigma)$ can have different number of connected components for different $\sigma$ (also $\bigoplus\limits_k$ is not defined there, presumably instead of $\bigoplus\limits_k C_{n+1}$ one should have $C_{n+1}^k$ for some $k\ge1$).
%It seems that the authors implicitly go back to the clarified (=without hocolimits) formulation of Proposition
%4.3 in order to obtain Proposition 4.3 (whose formulation involves hocolimits).
the $\pi$-preimage of the skeletal filtration $X^{(0)}\subset X^{(1)}\subset\ldots$ of $X$.
One obtains the spectral sequence \cite[\S20.2]{FF89} with
$$E^1_{p,q} = H_{p+q}(\pi^{-1}(X^{(p)},\pi^{-1}(X^{(p-1)})) \cong
%\bigoplus\limits_{\sigma\text{ a $p$-face of }X}
\oplus H_{p+q}(\sigma\times F_\sigma\cup R, \pi^{-1}(\partial\sigma)) \cong$$
$$\oplus H_{p+q}(\sigma\times F_\sigma,\partial \sigma\times F_\sigma) \cong\oplus H_q(F_\sigma).$$
Here the sums are over all $p$-faces $\sigma$ of $X$, and $R=\pi^{-1}(\partial\sigma)\backslash(\partial\sigma\times F_\sigma)$.
By the acyclicity we have
$$E^1_{p,0}=C_p(X)\quad\text{for}\quad p\le n,\quad
E^1_{p,q}=0\quad\text{for}\quad p+q\le n\quad\text{and}\quad q>0.$$
Since being $(-1)$-acyclic means being non-empty, the group $E^1_{n+1,0}$ has a quotient group $C_{n+1}(X)$.
The differential $E^1_{p,0}\to E^1_{p-1,0}$ is the boundary map $C_p(X)\to C_{p-1}(X)$ for $p\le n$,
and is the composition $E^1_{n+1,0}\to C_{n+1}(X)\to C_n(X)$ of the quotient and the boundary maps for $p=n+1$.
%Alternative: Hence  for $s\le n$ the group $\oplus_{p+q=s} E^2_{p,q}$ contains nothing but
%$E^2_{p,0} = H_p(X; H_0(C_{r-1}(X-\sigma)) H_p(X)$ in the bottom row.
%Therefore this group remains the same in subsequent pages of the spectral sequence.
Hence
$$E^2_{p,0}=H_p(X)\quad\text{for}\quad p\le n,\quad\text{and}\quad
E^2_{p,q}=0\quad\text{for}\quad p+q\le n\quad\text{and}\quad q>0.$$
Then $E^2_{p,q}=E^\infty_{p,q}$ for $p+q\le n$.
Therefore $H_s(X^{\underline r})\cong H_s(X)$ for every $s\le n$.
\end{proof}

%\[E^1_{p,q} = C_p(X; A_q),\quad E^2_{p,q} = H_p(X; A_q),\]
%where the system of coefficients over a $p$-face $\sigma\subset X$ is given by
%\[A_q(\sigma) = H_{q+p} (\pi^{-1}(\sigma), \pi^{-1}(\partial\sigma)) \cong
%H_p(\sigma,\partial\sigma)\otimes H_q (C_{r-1}(X-\sigma)) \cong H_q (C_{r-1}(X-\sigma)).\]
%Here $\cong$ are the K\"unneth isomorphism and the natural isomorphism defined by the orientation of $\sigma$.

%This DOES NOT follow from Vietoris--Begle mapping theorem:
%Recall the following Vietoris--Begle mapping theorem:
%\emph{Let $g:|P|\to |Q|$ be a surjective PL map between bodies of simplicial complexes whose every point-preimage is $n$-acyclic.
%Then $g_*:\widetilde H_j(P)\to \widetilde H_j(Q)$ is an isomorphism for any $j\le n$.  }

%(without using spectral sequences).
%(see references at \url{https://www.maths.ed.ac.uk/~v1ranick/papers/smale3.pdf}).
%https://en.wikipedia.org/wiki/Vietoris%E2%80%93Begle_mapping_theorem

\begin{remark}\label{r:hktt}
\aronly{(a)} In spite of being much shorter, the above proof of Theorem \ref{t:tvco} is not an alternative proof compared to \cite{HKTT} but is just a different exposition avoiding artificially sophisticated language, and applying \cite[Lemma]{vo96} instead of rewriting its proof in a slightly more sophisticated way.
%The part after Lemma 3.1 of \cite[\S3]{HKTT} is an artificially (but slightly and presumably unintentionally) %sophisticated exposition of the known proof of \cite[Lemma]{vo96}.
%Section 4 of \cite{HKTT} is an artificially (but presumably unintentionally) sophisticated exposition of the proof of Theorem \ref{t:tvco} presented above.
Analogously, the proofs of \cite{KM21} could be simplified (except possibly for the chirality result).

\aronly{
(b) In \cite[\S2, proof of Proposition 2.3]{HKTT} instead of  \cite[Lemma 2.2]{HKTT} one can use the following known result (similar to \cite[Lemma 2.2]{HKTT}):

\emph{If any intersection of some of $n$ complexes (including intersections involving only one complex)
is contractible (including empty), then $H^j$  of the union is zero for $j\ge n-1$.}

%This result is not stated in \cite[\S2]{HKTT}
%using Alexander duality .

(c) In \cite[end of the paragraph after Theorem 1.1]{HKTT} the phrase \emph{`Frick proved [13] that the condition that $r$ is a prime power is necessary'} is misleading as explained in \cite[Remark 1.9, \S3.1, \S5]{Sk16}.
The authors should have cited \cite{Sk16} (for an accurate description of references
%papers that might be referenced
concerning the counterexample to the topological Tverberg conjecture), and \cite{MW15}, \cite[\S3]{Sk16} (for the detailed version of the original proof, and for a published simplified survey exposition of the main ingredient to the counterexample).}
\end{remark}

%A slightly more sophisticated exposition in the paper under review does not make the proof essentially different from [Vo96, Lemma].
%\emph{For the spectral sequence associated to the preimage under the projection  $C_r(X)\to X$ of the skeletal filtration of $X$ we have
%$$E^1_{p,q}=\bigoplus\limits_{\sigma\text{ a $p$-face of }X}H_q(C_{r-1}(X-\sigma)),$$
%and this spectral sequence converges to (a group associated to) $H_*(C_r(X))$.}

A complex $A$ is said to be a \emph{refinement} of a complex $B$ if there is a homeomorphism $|A|\to|B|$ (called refining homeomorphism) between their bodies such that every face of $B$ is the union of images of some faces of $A$.

\begin{lemma}\label{l:subdiv}
(a) Any complex PL homeomorphic to the boundary $\partial\Delta_{n+1}$ of the simplex is a refinement of $\partial\Delta_{n+1}$.

(b) For any complex $D$ PL homeomorphic to $\Delta_n$ any refining homeomorphism
$h:\partial D\to\partial\Delta_n$ extends to a  homeomorphism $D\to\Delta_n$.
\end{lemma}

\begin{proof} (b) Take a PL homeomorphism $F:D\to\Delta_n$.
Denote by $cX$ the cone over a complex $X$.
Take a homeomorphism $C:\Delta_n\to c\partial\Delta_n$.
Then $C^{-1}\circ ch\circ cF|_{\partial D}\circ C\circ F$ is the required extension:
$$D \overset{F}\to \Delta_n \overset{C}\to c\partial\Delta_n \overset{cF|_{\partial D}}\to c\partial D \overset{ch}\to c\partial\Delta_n  \overset{C^{-1}}\to\Delta_n.$$

%both $D$ and $\Delta_n$ are PL homeomorphic to the cones over their boundaries.
%Then the conical extension of $h$ is a `refining' homeomorphism $D\to \Delta_n$ extending $h$.

(a) Induction on $n$.
The case $n=0$ is trivial.
Let us prove the inductive step $n-1\to n$.
Take any vertex $v$ of the complex.
The link $\lk v$ is PL homeomorphic to $\partial\Delta_n$.
So by the induction hypothesis $\lk v$ is a refinement of $\partial\Delta_n$.
Then $\st v$ is the conical refinement of the complement in $\partial\Delta_{n+1}$ to the interior of $\Delta_n$.
Denote by $D$ the complement in the given complex to the open star of $v$.
By \cite[3.13]{RS72} $D$ is homeomorphic to $\Delta_n$.
So by (b) the refining homeomorphism $\lk v\to\partial\Delta_n$ extends to a homeomorphism $D\to\Delta_n$, which is automatically refining.\footnote{The above proofs of (a,b) are analogous to \cite[p. 200]{Gr03} (they use \cite[3.13]{RS72} instead of  central projection).
\aronly{\newline
It was noted in \cite[before question 1.2]{HKTT} that \emph{if $r$ is a prime power, then no boundary of a convex $(d+1)(r-1)$-polytope has an almost $r$-embeddings to $\R^d$}, and that this immediately follows from the topological Tverberg theorem because \cite[p. 200]{Gr03} \emph{the boundary of a convex $(n+1)$-polytope is a refinement of $\partial\Delta_{n+1}$.}
\newline
Each one of (a), (b) implies that \emph{any complex $D$ PL homeomorphic to $\Delta_n$ is a refinement of $\Delta_n$} (absolute version of (b)).
In the above proof of (a) we need (b), not its absolute version.}}
%This refinement clearly extends the refinement $\lk v$ of $\partial\Delta_n$.
%So the inductive step follows.
\end{proof}

%(c) In \cite[before question 1.2]{HKTT} instead of `the answer is positive' there should be `the answer is positive for $r$ a prime power'.


\begin{thebibliography}{RSS95}

%\input{references}

%\bz

\bibitem[BZ16]{BZ16} * \emph{P. V. M. Blagojevi\'c and G. M. Ziegler,} Beyond the Borsuk-Ulam theorem: The topological Tverberg story, in: A Journey Through Discrete Mathematics, Eds. M. Loebl,
J. Ne\v set\v ril, R. Thomas, Springer, 2017, 273--341. arXiv:1605.07321v3.

%\ffene

\bibitem[FF89]{FF89} * \emph{A.T. Fomenko and D.B. Fuchs.} Homotopical Topology, Springer, 2016.

\bibitem[Gr03]{Gr03} * \emph{B. Gr\"unbaum.} Convex polytopes, 2nd ed., Graduate Texts in Mathematics, vol. 221, Springer, New York, 2003.

\bibitem[HKTT]{HKTT} \emph{S. Hasui, D. Kishimoto, M. Takeda, M. Tsutaya.}  Tverberg's theorem for cell complexes, Bull. LMS, 55:4 (2023), 1944--1956. arXiv:2101.10596.

\bibitem[KM21]{KM21} \emph{D. Kishimoto, T. Matsushita.} Van Kampen-Flores theorem for cell complexes, Discrete Comput. Geom., to appear. arXiv:2109.09919.

\aronly{
%\mwof
\bibitem[MW15]{MW15} \emph{I. Mabillard and U. Wagner.}
Eliminating Higher-Multiplicity Intersections, I. A Whitney Trick for Tverberg-Type Problems. arXiv:1508.02349.}

%\rsst

\bibitem[RS72]{RS72} * \emph{C. P. Rourke and B. J. Sanderson,}
\newblock Introduction to Piecewise-Linear Topology,
\newblock \emph{Ergebn.\ der Math.} 69, Springer-Verlag, Berlin, 1972.

%\skos

\bibitem[Sk16]{Sk16} * \emph{A. Skopenkov,} A user's guide to the topological Tverberg Conjecture, arXiv:1605.05141v5. Abridged earlier published version: Russian Math. Surveys, 73:2 (2018), 323--353.

\bibitem[SZ24]{SZ24} \emph{P. Sober\'on, S. Zerbib.} The B\'ar\'any-Kalai conjecture for certain families of
    polytopes, arXiv:2404.11533.

%psoberon@gc.cuny.edu, zerbib@iastate.edu

%\vo

\bibitem[Vo96]{vo96} \emph{A. Yu. Volovikov,} On a topological generalization of the Tverberg theorem. Math. Notes 59:3 (1996), 324--326.

%alvostef@mail.ru

{\it Books, surveys and expository papers in this list are marked by the stars.}

\end{thebibliography}
\end{document}